\documentclass{amsart}
\usepackage{amssymb,amsmath,amsrefs,graphicx}

\newtheorem{theorem}{Theorem}[section]
\newtheorem{lemma}[theorem]{Lemma}
\newtheorem{corollary}[theorem]{Corollary}
\newtheorem{proposition}[theorem]{Proposition}

\newtheorem{definition}{Definition}

\newtheorem{remark}{Remark}

\newtheorem{example}{Example}

\DeclareMathOperator{\ad}{ad}
\DeclareMathOperator{\Z}{Z}
\DeclareMathOperator{\Nil}{Nil}
\DeclareMathOperator{\Rad}{Rad}
\DeclareMathOperator{\Asoc}{Asoc}
\DeclareMathOperator{\Soc}{Soc}
\DeclareMathOperator{\Ssoc}{Ssoc}

\DeclareMathOperator{\Sub}{Sub}
\DeclareMathOperator{\Pow}{Pow}
\DeclareMathOperator{\degr}{degr}
\DeclareMathOperator{\id}{id}
  
\begin{document}
\title{Lie algebras with a finite number of ideals}
\author{Pilar Benito and Jorge Rold\'an-L\'opez}
\address{Dpto. Matem\'aticas y Computaci\'on,
  Universidad de La Rioja, 26006,
  Logro\~no, La Rioja, Spain}
\email{pilar.benito@unirioja.es and jorge.roldanl@unirioja.es}
\date{\today}
\thanks{The authors are partially funded by grant MTM2017-83506-C2-1-P of Ministerio de Econom\'ia, Industria y Competititividad (Spain). The second-named author is supported by a predoctoral research grant of the Universidad of La Rioja.}
\begin{abstract}
  In this paper we focus on the structure of the variety of Lie algebras with a finite number of ideals and their graph representations using Hasse diagrams. The large number of necessary conditions on the algebraic structure of this type of algebras leads to the explicit description of those algebras in the variety with trivial Frattini subalgebra. To illustrate our results, we have included and discussed lots of examples throughout the paper.
\end{abstract}
\maketitle

\section {Introduction}

The set of ideals of a Lie algebra $L$ is a \emph{poset} (partially ordered set) ordered by inclusion. In the set of ideals of $L$, every two elements have supremum (also called least upper bound or join) and infimum (greatest lower bound or meet). From their definitions, it is easily checked that the supremum and infimum are unique. The supremum of two ideals is given by the sum of the ideals and the infimum is just the intersection. So, the poset of ideals of $L$ is a lattice (see~\cite{Gr11} for a formal definition). It is worth pointing out that lattices of ideals of Lie algebras are complete lattices because every subset of ideals has a join  and a meet. Moreover, this lattices also are bounded: the zero ideal is the smallest element and the total algebra is the largest element. Many fundamental properties of Lie algebras can be interpreted as facts about ideal lattices. For instance, lattices of ideals of Lie algebras that have a nondegenerate invariant bilinear form are \emph{selfdual} as it is pointed out in \cite{HoKe86}*{Section 1}.

A Hasse diagram represents a finite poset in the form of a drawing of its transitive reduction. For a poset $(S, \leq)$, the Hasse diagram represents the elements of $S$ by nodes (small black circles in our pictures). The nodes representing the elements $x$ and $y$ of $S$ are connected by a straight line or segment that goes upward from $x$ to $y$ whenever $y$ \emph{covers} $x$, that is, whenever $x < y$ and there is no $z$ such that $x < z < y$. These segments may cross each other but must not touch any vertices other than their endpoints. We have drawn some examples of these diagrams in Figure \ref{fig:hasseExamples}.

\begin{figure}
  \centering
  \includegraphics{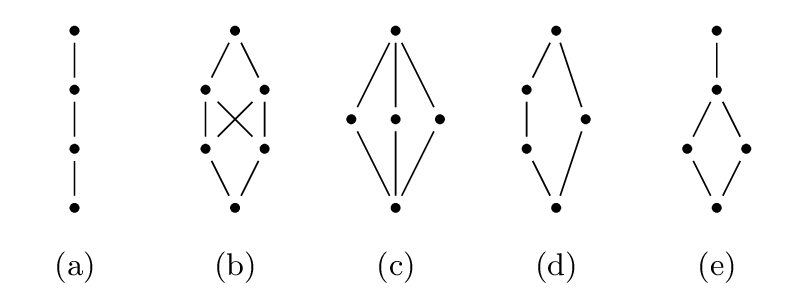}
  \caption{Examples of different Hasse digrams.}
  \label{fig:hasseExamples}
\end{figure}

According to~\cite{Gr11}*{Lemma 1, Section 1.4}, from a Hasse diagram we can recapture the relation $\leq $ by noting that $x<y$ holds iff there exists a sequence of elements $c_0,\dots, c_n$ such that $x=c_0$, $y=c_n$ and $c_{i+1}$ covers $c_i$. Hence, the Hasse diagram of a finite poset determines the poset up to isomorphisms. A finite lattice is a poset attached to a Hasse diagram for which every pair of nodes has a supremum and an infimum. The above remark implies that for each of the diagrams in Figure \ref{fig:hasseExamples}, the corresponding poset is a lattice with the exception of diagram (b).

The aim of this work is to bring together the algebraic and lattice properties of finite dimensional Lie algebras with a finite number of ideals. Starting out with semisimple Lie algebras as main examples of this class of algebras, we will prove that finite lattices that can occur as lattices of ideals of Lie algebras satisfy distributive laws. Hence, diagrams (a) and (e) in Figure \ref{fig:hasseExamples} are the only ones that could be (and in fact they are) lattices of ideals of Lie algebras.

The paper is organized into four sections from number $2$. In Section~2 we set out the basic terminology and facts on lattices according to Gr\"{a}tzer \cite{Gr11}. In this section we establish that finite lattices of ideals are distributive lattices. We also provide alternative characterizations of
the variety of Lie algebras with boolean lattices. Section~3 includes several necessary conditions of the algebraic structure of the class of Lie algebras with a finite number of ideals. The results in this section will be applied in the final Section 4 to provide the complete description of the subclass of algebras with trivial Frattini subalgebra. This section ended with a detailed outline of Lie algebras with trivial Frattini subalgebra up to 10 ideals and the corresponding Hasse diagrams.

Throughout the study we will follow ideas and techniques given in \cite{Be95}. Vector spaces along the paper are finite-dimensional over any field $\mathbb{F}$ of characteristic zero. Direct sums as vector spaces will be denoted by $V\oplus W$. Basic definitions and structure results on Lie algebras  can be found in \cite{Hu72} or \cite{Ja62}.

\section{Lattices, finite lattices of ideals and Hasse diagrams}

The algebraic structure of a Lie algebra imposes strong conditions on its lattice of ideals. Unfortunately, the ideal lattice structure does not always determine the Lie algebra in a unique way. For instance, the one-dimensional Lie algebra and every simple Lie algebra have the same lattice of ideals consisting of a 2-element chain. This elementary example shows that there exist non isomorphic Lie algebras with the same lattice of ideals. Nevertheless, the lattice structure of a Lie algebra does approach to the algebraic structure. In fact, it is quite natural to apply theoretical lattice ideas such us complementation, modularity or distributivity in the study of the ideals of a Lie algebra to determine structural properties of the algebra. In this section we will illustrate this assertion.

From~\cite{Gr11}*{Section 1.8, Chapter I}, starting out with a poset $(S, \leq)$, we introduce the notion of infimum and supremum of any subset of $S$ as it is usually done for (infinite) sets of real numbers. The poset $(S, \leq)$ is a lattice if $a\vee b=\sup\{a,c\}$ (join) and $a\wedge b=\inf\{a,b\}$ (meet) exist for all $a,b\in S$ and they are unique. If the lattice has largest (usually denoted as $1$) and smallest (denoted as $0$) elements, it is called bounded. We say that the lattice is complete if $\bigvee H$ and $\bigwedge H$ exist for any subset $H\subseteq S$.

As we have mentioned in the introduction, the poset of the ideals of a Lie algebra $L$ is a bounded and complete lattice where $L$ is the least upper bound and $\{0\}$ the greatest lower bound. In addition, for any three ideals $A,B,C$ of $L$ such that $A\subseteq B$, we have the identity
\begin{equation}\label{eq:modular}
  B\cap(A+C)=A+(B\cap C).
\end{equation}
In a general lattice, (\ref{eq:modular}) can be rewritten as,
$$
  a<b \quad \textrm{implies
  }\quad \inf\{b,\sup\{a,c\}\}=\inf\{b,\sup\{a,c\}\}$$ for arbitrary $a,b,c$  elements of the lattice. This identity is known as \emph{Modular Law} and it is the identity that defines modular lattices. So, lattices of ideals of Lie algebras are modular.

The following definition introduces the different classes of lattices that will appear in the paper. Any element that covers $0$ in a lattice is an \emph{atom}.

\begin{definition}
  Let $\mathcal{L}$ be a bounded lattice attached to the poset $(S,\leq)$ with join and meet denoted by $\vee$ and $\wedge$ respectively. Then:
  \begin{itemize}
    \item[\emph{(a)}] $\mathcal{L}$ is a complemented lattice if each element has a complement, that is, for a given element $a$, there is an element $b$ such that $a\vee b=1$ and $a\wedge b=0$.
    \item[\emph{(b)}] $\mathcal{L}$ is a distributive lattice if $\mathcal{L}$ satisfies either of the following equivalent distributive laws:
          \begin{itemize}
            \item[] \emph{(D1)} $\quad a\vee(b\wedge c)=(a\vee b)\wedge(a\vee c)$
            \item[] \emph{(D2)} $\quad a\wedge(b\vee c)=(a\wedge b)\vee(a\wedge c)$
          \end{itemize}
    \item[\emph{(c)}] $\mathcal{L}$ is a boolean lattice if it is complemented and distributive.
    \item[\emph{(d)}] $\mathcal{L}$ satisfies the \emph{diamond property} \emph{(DP)} if for two distinct atoms $a,b$ there exists an element $c$ for which $\{0,a,b,c,a\vee b\}$ is the $M_3$ lattice in Figure~\ref{fig:nomodular}.
  \end{itemize}
\end{definition}

The two typical examples of nondistributive lattices are the pentagon $N_5$ and the diamond $M_3$ whose Hasse diagrams are shown in Figure~\ref{fig:nomodular}. 
The following theorem characterizes distributivite and modular lattices by means of pentagons and diamonds.

\begin{figure}[h]
  \centering
  \includegraphics{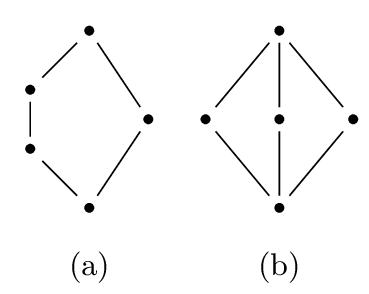}
  \caption{(a) is the $N_5$ non modular lattice, and (b) the $M_3$ nondistributive lattice.}
  \label{fig:nomodular}
\end{figure}

\begin{theorem}\label{distributive-modular}\emph{\cite{Gr11}*{Theorem 102}} Every distributive lattice is modular. Moreover:
  \begin{itemize}
    \item[\emph{(a)}]A lattice is modular iff does not contain a pentagon.
    \item[\emph{(b)}]A lattice is distributive iff does not contain a pentagon or a diamond. \hfill $\square$
  \end{itemize}
\end{theorem}

In a totally ordered set $(S,\leq)$, any pair of elements are comparable. Thus the infimum and supremum of any pair of elements of $S$ exist and therefore $S$ is a lattice. If $S$ is a set of $n\geq 1$ elements, we will say that $S$ is an $n$-element chain lattice ($n$-chain for short). According to the introduction, a finite lattice can be represented by the Hasse diagram of its underlying poset. In this way, the Hasse diagram for any $n$-chain is given in Figure~\ref{fig:nHase} (the diagram has exactly $n$ nodes). A general lattice is said to be of length $k$, where $k$ is a natural number ($k=0$ is possible), if there is a $(k+1)$-chain sublattice and all of its chain sublattices have a number of elements $\leq k+1$. Note that the length of an $n$-chain is $n-1$, just the number of ``jumps''.

\begin{figure}[h]
  \centering
  \includegraphics{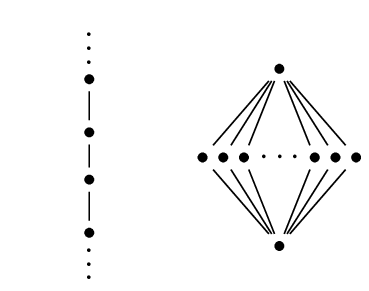}
  \caption{On the left the $n$-chain, on the right the $M_n$ lattice.}
  \label{fig:nHase}
\end{figure}

The diamond lattice $M_3$ generalizes to $M_n$ which is the $2$-length lattice with $n\geq 3$ atoms. The Hasse diagram of $M_n$ is given in Figure~\ref{fig:nHase}. The set of subspaces, $\Sub \mathbb{F}^n$, of the $n$-dimensional vector space $\mathbb{F}^n$ and the power set $\Pow A$, i.e. the set of all subsets of the set $A$, are other examples of lattices. Both $\Sub \mathbb{F}^n$ and $\Pow A$ are posets ordered by inclusion. The former is a lattice with meet and join given by intersection and sum and the latter under intersection and union.

The following examples relate the set of ideals of some varieties of Lie algebras and the previous notions and examples of lattices. From now on, $\mathcal{L}_{\id}(L)$ will denote the lattice of ideals of a Lie algebra $L$ with (bracket) product $[x,y]$.

\begin{example} \emph{$n$-chain lattice $C_n (n\geq 1)$:} This lattice is distributive and noncomplemented if $n\geq 3$. The trivial Lie algebra is the only algebra whose lattice of ideals is $C_1$. The variety of Lie algebras that have $C_2$ as lattice of ideals are the set of simple Lie algebras plus the $1$-dimensional algebra. For any $n\geq 3$, Lie algebras whose lattice of ideals is an $n$-chain have been fully characterized in~\emph{\cite{Be92_2}}. A complete classification is given for solvable Lie algebras over algebraically closed fields.
\end{example}

\begin{example}\label{n-subspace} \emph{$n$-subspace lattice $\Sub\ \mathbb{F}^n$ ($n\geq 0$):} This is a complemented and modular lattice of length $n$ that satisfies the property \emph{(DP)}. This lattice is not finite and nondistributive if $n\geq 2$. The variety of Lie algebras whose lattice of ideals is isomorphic to $\Sub \mathbb{F}^n$ for some $n$, is the variety of finite dimensional abelian Lie algebras ($L^2=[L,L]=0$): Any algebra $L$ such that $\mathcal{L}_{\id}(L)$ is isomorphic to $Sub\ \mathbb{F}^n$ decomposes into the sum of $n$ minimal ideals $L=I_1\oplus \dots \oplus I_n$. The ideals in the decomposition are abelian because of the \emph{(DP)}-property, and therefore $L^2=0$.
\end{example}

\begin{example}\label{Mm} \emph{$M_n$ lattice ($n\geq 3$):} This lattice has length $2$ and it is complemented, modular and satisfies the \emph{(DP)}-property, so it is nondistributive. There are no Lie algebras with lattice of ideals isomorphic to $M_n$: If $\mathcal{L}_{\id}(L)$ is isomorphic to $M_n$, \emph{(DP)}-property implies $L^2=0$, so $L$ is a $2$-dimensional abelian Lie algebra and therefore $\mathcal{L}_{\id}(L)$ is not a finite lattice, a contradiction.
\end{example}

For a finite set $A$, the Hasse diagram of the lattice of $\Pow A=2^A$ has $2^{\mid A\mid}$ nodes. Note that a line joints two nodes whenever the corresponding subsets differ in a single element. In this way we arrive at the hypercube lattice or $n$-cube lattice $Q_n$, the lattice of subsets of the set $\{1,2,\dots, n\}$. The $0$-cube $Q_0$ is the lattice of subsets of the empty set. We point out the following result about $n$-cube lattices:

\begin{theorem}\label{boolean}\emph{\cite{Gr11}*{Corollaries 109 and 110}} Let $\mathcal{L}$ be a finite lattice. Then:
  \begin{itemize}
    \item[\emph{(a)}]$\mathcal{L}$ is distributive iff $\mathcal{L}$ is isomorphic to a sublattice of the lattice of subsets of some finite set.
    \item[\emph{(b)}]$\mathcal{L}$ is boolean iff $\mathcal{L}$ is isomorphic to $Q_n$ for some $n\geq 0$. \hfill $\square$
  \end{itemize}
\end{theorem}

\begin{example} \emph{$n$-cube lattice $Q_n\, (n\geq 0)$:}  According to Corollary \ref{cor-booleano}, the reductive Lie algebras with centre of dimension up to one is the variety of Lie algebras with hypercube lattices. In Figure \ref{fig:n-chain} we have drawn the $n$-cube lattices of ideals of either a semisimple Lie algebra with $n=1,2,3,4,5$ simple components or $L=S\oplus \mathbb{F}z$, where $\mathbb{F}z$ is the centre of $L$ and $S$ is a semisimple ideal with $n-1=0,1,2,3,4$ simple components. The lattices $Q_4$ and $Q_5$ are called tesseract and penteract respectively.

  \begin{figure}
    \centering
    \includegraphics[width=\textwidth]{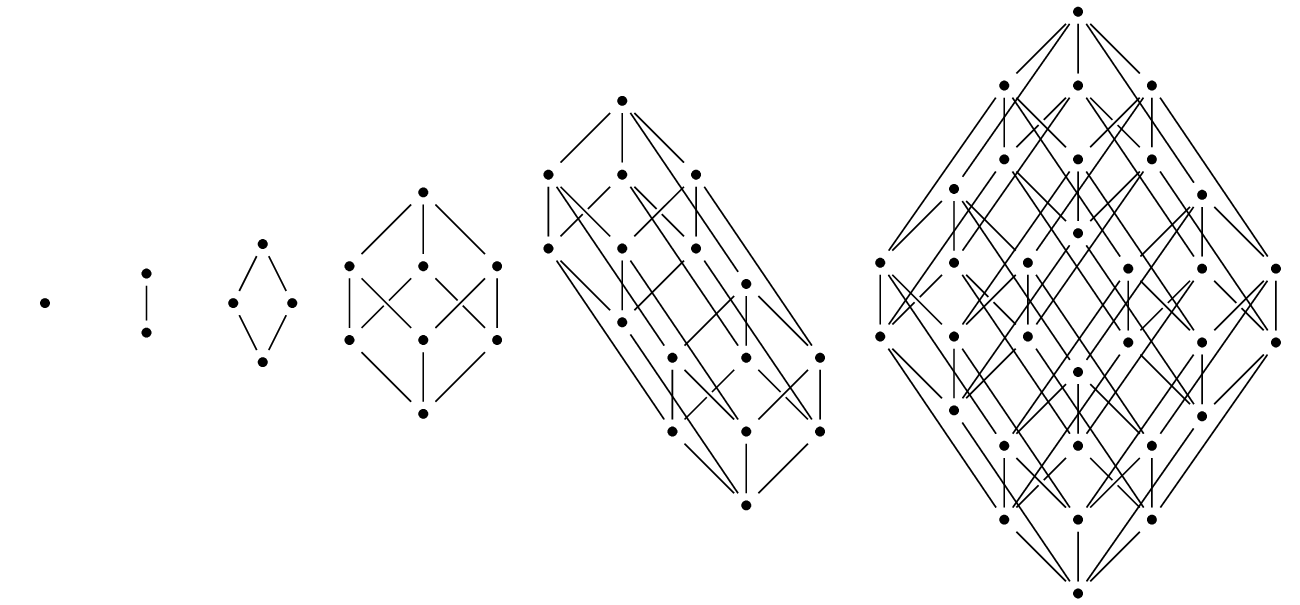}
    \caption{$n$-cube lattices from $n=0$ (left) to $n=5$ (right).}
    \label{fig:n-chain}
  \end{figure}
\end{example}

All the lattices in previous examples are selfdual. In the case of finite lattices, the dualization turns the Hasse diagram upside down.

For a given vector space $M$ and any subset $\Omega$ of linear maps in $M$, the set of subspaces of $M$ which are invariant under $\Omega$ is a lattice where join and meet are given by the sum and intersection of invariant subspaces. We will denote this lattice as $\Sub_\Omega M$. The two subspaces $A,B \in \Sub_\Omega M$ are $\Omega$-isomorphic if there is a bijective map $\varphi\colon A\to B$ such that $\varphi(f(a))=f(\varphi(a))$ for any $a\in A$ and any $f\in \Omega$.

\begin{lemma}\label{completely-reducible}
  Let $\Omega$ be a Lie algebra of linear transformations in a vector space $M$ and let $\Sub_\Omega M$ be the lattice of $\Omega$-invariant subspaces of $M$. Then, the following assertions are equivalent:
  \begin{itemize}
    \item [\emph{(a)}] $\Sub_\Omega M$ is a finite and boolean lattice.
    \item [\emph{(b)}] $\Sub_\Omega M$ is a finite and complemented lattice.
    \item [\emph{(c)}] $\Omega$ is completely reducible in $M$ and $\Sub_\Omega M$ has no irreducible $\Omega$-isomorphic elements.
    \item [\emph{(d)}] $\Omega=\Omega_1\oplus A$, where $\Omega_1$ is a semisimple ideal of $\Omega$, $A$ is the centre, the elements of $A$ are semisimple, i.e. the minimum polynomial of any linear map of $A$ is the product of relatively prime irreducible polynomials, and $\Sub_\Omega M$ has no irreducible $\Omega$-isomorphic elements.
  \end{itemize}

  \noindent If any of these equivalent conditions hold, $M$ has a unique decomposition $M=M_1\oplus\dots \oplus M_r$, where $M_j\in \Sub_\Omega M$ are irreducible and no $\Omega$-isomorphic subspaces. In particular,$$
    \Sub_\Omega M=\{M_{i_1}\oplus\dots \oplus M_{i_k}: 1\leq i_1<\dots <i_k\leq r\}\cup\{0\},
  $$and $\Sub_\Omega M$ is isomorphic to the $r$-cube lattice $Q_r$.
\end{lemma}

\begin{proof}
  If we omit the condition of finiteness and the property of nonexistence of $\Omega$-isomorphic subspaces, the equivalence of the assertions (b), (c) and (d) follows from~\cite{Ja62} Theorem 9 in Section 5 of Chapter II, and Theorem 10 in Section 7 of Chapter III.

  Now assume $\Sub_\Omega M$ is a finite and complemented lattice. Then, there are no irreducible $\Omega$-isomorphic subspaces $P$ and $Q$. Otherwise, there exists a $\Omega$-isomorphism $\varphi\colon P\to Q$, so $\varphi f=f \varphi$ for any $f\in \Omega$. In this case, for any $ \alpha\in \mathbb{F}$ we can define the nonzero subspace,
  $$
    R^\alpha = \{a+\alpha\,\varphi(a) \mid a\in P\}\subseteq P\oplus Q.
  $$It is clear that $R^\alpha \in \Sub_\Omega M$. Suppose $R^\alpha =R^\beta $, so $a+\alpha\, \varphi(a)=b+\beta\, \varphi(b)$, $a,b \in P\setminus\{0\}$. The last equation implies
  \begin{equation}
    a - b = \beta\,\varphi(b)-\alpha\,\varphi(a)=\varphi (\alpha a-\beta b) \in P \cap Q = 0,
  \end{equation}
  and therefore $\beta=\alpha$. Hence for any scalar $\alpha$ we get distinct $\Omega$-invariant subspaces $R^\alpha$, a contradiction because the based field $\mathbb{F}$ is infinite and $\Sub_\Omega M$ is a finite set.

  Assume finally that $\Omega$ is completely reducible in $M$ and $M$ does not contain  $\Omega$-isomorphic irreducible subspaces. Then $M=M_1\oplus\dots \oplus M_r$, where the different $M_j\in \Sub_\Omega M$ are not $\Omega$-isomorphic. For any irreducible subspace $P$, it follows that $M=P\oplus P_2\oplus \dots \oplus P_k$, where $P_i$ are irreducible. Since there are no $\Omega$-isomorphic irreducible subspaces, applying Krull-Schmidt Theorem of sets of linear transformations, we get that $k+1=r$ and $P=M_{j}$ for some $1\leq j\leq r$. So, apart from the trivial subspace, the $\Omega$-invariant subspaces of $M$ are direct sums of a finite number of subspaces $M_1,\dots, M_r$.
  The equivalence of (a) and (c) follows from~\cite{Pi82}*{Chapter 2, Corollary c, Section 2.4} and therefore Theorem~\ref{boolean} ends the proof.
\end{proof}

\begin{theorem}\label{finite-lattices}
  Finite lattices of ideals of Lie algebras are always distributive. Up to isomorphisms, these lattices are sublattices of $n$-cube lattices.
\end{theorem}

\begin{proof}
  Let $L$ be a Lie algebra and $\mathcal{L}=\mathcal{L}_{\id}(L)$ be its lattice of ideals. Assume the result is false. Since $\mathcal{L}$ is a modular and nondistributive lattice, there is a sublattice of type $M_3$ in $\mathcal{L}$ according to Theorem~\ref{distributive-modular}. So there exist ideals $K,R, P_1, P_2$ and $P_3$ such that $P_i\cap P_j=K$, $P_i+P_j=R$, and $R$ covers $P_i$ and $P_i$ covers $K$ for each $i=1,2,3$. Consider now the quotient Lie algebra $L/K$ and note that $Q_i=P_i/K$ are minimal ideals. Since $[P_i,P_j]\subseteq K$ and $P_k\subseteq P_i+P_j=R$, the ideal $T=R/K$ decomposes as the direct sums, $T=Q_1\oplus Q_2=Q_1\oplus Q_3=Q_2\oplus Q_3$. Then $[Q_i,Q_j]\subseteq Q_i\cap Q_j=0$, and therefore each $Q_i$ is an abelian ideal. Let $M$ denote the sum of all minimal abelian ideals of $L/K$, and consider the Lie algebra of linear maps in $M$, $\Omega=\ad_M L/K$, $\ad x(m)=[x,m]$ for all $x\in L/K, m\in M$. The elements of the lattice $\Sub_\Omega M$ are just the set of ideals of $L/K$ inside $M$. This lattice is finite and complemented because of $M$ decomposes as a direct sum of minimal abelian ideals. Using Lemma~\ref{completely-reducible}, there is no $\Omega$-isomorphic ideals of $L$ contained in $M$. Since $Q_i$ are ideals, every canonical projection $\varphi_{ijk}\colon T=Q_i\oplus Q_j\to Q_k$, $\{(i,j,k): i<j, k=i,j\}$, satisfies $\varphi_{ijk}([x,a])=[x,\varphi_{ijk}(a)]$ for all $x\in L/K$. So, in a natural way, the maps $\varphi_{ijk}$ let us define an $\Omega$-isomorphism $\varphi\colon Q_1\to Q_2$, a contradiction. The final part follows from Theorem~\ref{boolean}.
\end{proof}

The next result describes the variety of Lie algebras whose lattice of ideals is either complemented or boolean. The Jacobson radical of a Lie algebra $L$ is the intersection of all maximal ideals of $L$ and $Z(L)$ stands for the centre of $L$ (so $[L,Z(L)]=0$). We remark that the general assertion about complemented lattices and the equivalence of statements (a) and (e) have been stablished in Lemma 2.3 in~\cite{Be95}.

\begin{corollary}\label{cor-booleano}
  The Lie algebras with complemented lattice of ideals are of the form $L=S\oplus A$, where $S$ is a semisimple ideal of $L$ and $A$ is an abelian ideal, so $A=Z(L)$. Moreover, the following assertions are equivalents:
  \begin{itemize}
    \item [\emph{(a)}]$\mathcal{L}_{\id}(L)$ is a boolean lattice.
    \item [\emph{(b)}]$\mathcal{L}_{\id}(L)$ is finite and complemented.
    \item [\emph{(c)}]$\mathcal{L}_{\id}(L)$ is an $n$-cube lattice for some $n\geq 0$.
    \item [\emph{(d)}]The Jacobson radical of $L$ is trivial and $Z(L)$ has dimension at most 1.
    \item [\emph{(e)}]$L$ is either $0$ or one of the following Lie algebras: $\mathbb{F}z$, a semisimple algebra $S$ or a direct sum as ideals of $\mathbb{F}z$ and $S$.

  \end{itemize}
  In this case, $L$ has $2^n$ ideals where $n$ is either $r$ or $r+1$ and $r$ is the number of simple components of $S$ ($r=0$ is also possible).
\end{corollary}
\begin{proof}
  Let $\Omega=\ad L=\{\ad x:x\in L\}$, where $\ad x(a)=[x,a]$. Note that $\mathcal{L}_{\id}(L)$ is the set of $\Omega$-invariant subspaces of $L$, so $\mathcal{L}_{\id}(L)=\Sub_\Omega L$. We will prove firstly the characterization of complemented lattices. Assume $\mathcal{L}_{\id}(L)$ is complemented and let $A$ be an ideal such that $L=L^2\oplus A$. Then $[L,A]\subseteq L^2\cap A=0$, thus $A\subset \Z(L)$. On the other hand, the derived ideal $L^2$ decomposes as a sum of minimal ideals, $L^2=I_1\oplus \dots\oplus I_k$, and $I_j^2=I_j$ because of $(L^2)^2=L^2$. Hence $L=L^2\oplus \Z(L)$ is a direct sum of the semisimple ideal $L^2$ and the centre. Conversely, in the case $L=S\oplus Z(L)$, $\Omega=\ad L\cong \ad S$ and $\mathcal{L}_{\id}(L)$ is complemented by \cite{Ja62}*{Theorem 10, Section 7, Chapter III}.

  Now we will check the equivalence of the five conditions. Theorem~\ref{finite-lattices} shows that (b) implies (a). The previous paragraph and Example~\ref{n-subspace} give us (e) from (a). Since the Jacobson radical of $L$ is just $L^2\cap \Rad(L)=[L,\Rad(L)]$, where  $\Rad(L)$ is the solvable radical (see \cite{Ja62}*{Chapter III, Section 9} and \cite{Ma67}), (e)$\Rightarrow$(d) follows from $[L,\Rad(L)]=[L,\mathbb{F}z]=0$. In the case $L^2\cap \Rad(L)=0$ we have $L=S\oplus \Z(L)$, so $S=L^2$ is semisimple and $\ad L\cong S$. Then (d)$\Rightarrow$(b) follows by using $\dim Z(L)\leq 1$ and Lemma~\ref{completely-reducible}. Finally, Theorem~\ref{boolean} ensures that (a) and (c) are equivalent.
\end{proof}

\begin{remark}\label{equivalence reductive}
  We say a Lie algebra $L$ is reductive if the adjoint representation $\ad\colon L\to \mathfrak{gl}(L)$ is completely reducible (see \emph{\cite{Ja62}*{Chapter III, Exercise 19}}). In \emph{\cite{Hu72}*{Chapter V, Section 19}}, a Lie algebra is called reductive if the solvable radical is just the centre. In the literature, reductive Lie algebras are also defined as the variety of Lie algebras with trivial Jacobson radical. All these definitions are equivalent.
\end{remark}

\begin{remark}
  Different algorithms to obtain modular or distributive finite lattices can be found in \emph{\cite{EHR02}} and  \emph{\cite{JiLa15}}. In Figure~\ref{fig:allHasse} we show all possible distributive lattices up to $8$ nodes according to \emph{\cite{EHR02}}. Lie algebras up to 6 ideals have been characterized in \emph{\cite{Be92_1}} and \emph{\cite{Ro17}}. Using ideas and techniques included in \emph{\cite{Be92_1}}, \emph{\cite{Be95}} and \emph{\cite{Ro17}}, most of the distributive lattices of 7 and 8 nodes can be easily performed as the lattices of ideals of some Lie algebra.
\end{remark}

\begin{figure}
  \centering
  \includegraphics[width=\textwidth,keepaspectratio,height=0.9\textheight]{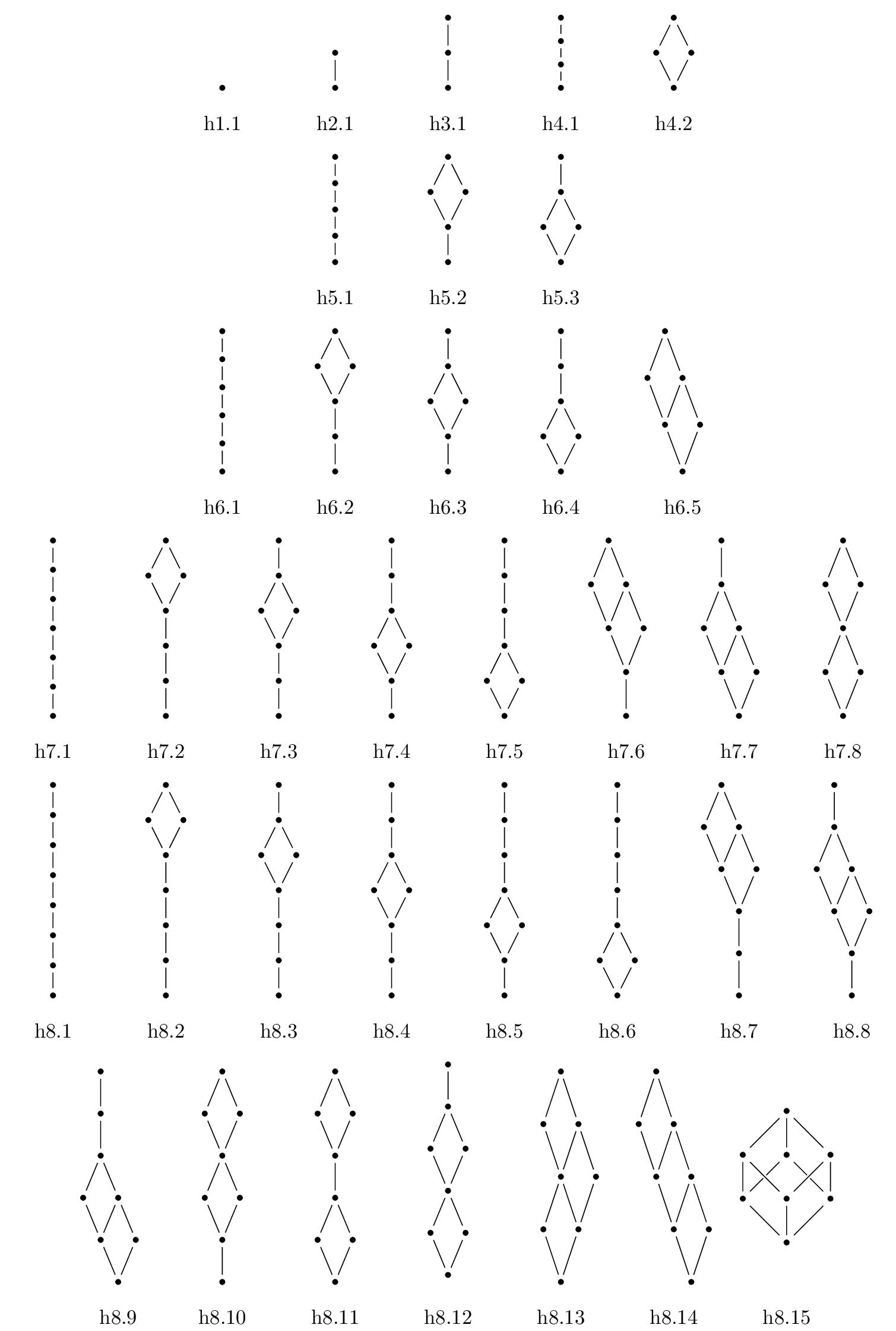}
  \caption{All Lie algebras Hasse diagrams up to 8 ideals (nodes).}
  \label{fig:allHasse}
\end{figure}

\section {Preliminary structure results}

Let $L$ be a Lie algebra. From Levi's Theorem~\cite{Ja62}*{Chapter III, Section 9}, $L$ decomposes as $L=S\oplus \Rad(L)$ where $S$ is a (maximal) semisimple subalgebra known as Levi factor of $L$, and $\Rad(L)$ is the solvable radical. This decomposition is called \emph{Levi decomposition} of $L$. The restricted adjoint representation:
\begin{equation}\label{eq:adjointrepresentation}
  \rho\colon S\to \mathfrak{gl}(\Rad(L)),\, \rho(s)(a)=\ad s(a)=[s,a],
\end{equation}
induces a $S$-module structure on $\Rad(L)$. So the solvable radical decomposes (not uniquely) as a direct sum of irreducible $S$-modules, $\ad S$-modules or $\ad S$-subspaces in the sequel. The algebra $L$ is called \emph{faithful} if so is the adjoint representation $\rho$. This condition is equivalent to the fact that $L$ has no simple ideals. Denote by $\Ssoc(L)$ the sum of all simple ideals of $L$. The ideal $\Ssoc(L)$ is a semisimple algebra which is contained in any Levi factor of $L$. Then, we can decompose $L$ as $L=\Ssoc(L)\oplus S'\oplus \Rad(L)$ where $S=\Ssoc(L)\oplus S'$, $L^*=S'\oplus \Rad(L)$ is a faithful subalgebra of $L$, $S'$ is a Levi factor of $L^*$ and $\Rad(L^*)=\Rad(L)$.

\begin{lemma}\label{lema-reduction-1} The ideals of a Lie algebra $L$ with Levi decomposition $L=S\oplus \Rad(L)$ are of the form $B\oplus A$, where $A$ is an ideal of $L$ contained in $\Rad(L)$ and $B$ is an ideal of $S$ such that $[B,\Rad(L)]\subseteq A$.
\end{lemma}
\begin{proof}
  For such $B$ and $A$, it is obvious that  the vector space $B+A$ is an ideal of $L$. 
  Assume then that $I$ is any ideal and consider the Levi decomposition of $I$, so $I=I(S)\oplus \Rad(I)$ where $I(S)$ is a semisimple subalgebra of $L$. It is well known that $\Rad(I)=\Rad(L)\cap I$, which is an ideal of $L$ contained in $\Rad(L)$. On the other hand, according to the Malcev-Harish-Chandra Therorem~\cite{Ja62}*{Chapter III, Section 9}, there exists an inner authomorphism $\sigma$ of $L$ such that $\sigma (I(S))\subseteq S$. Since $I$ is an ideal of $L$, $\sigma (I(S))$ is also contained in $I$. Hence $\sigma (I(S))\subseteq S\cap I$. Now $\sigma (I(S))$ is a Levi factor of $I$ and $S\cap I$  a semisimple subalgebra, so $\sigma(I(S))=S\cap I$. Therefore, $I=B\oplus A$, where $B=S\cap I$ is an ideal of $S$, $A=\Rad(L)\cap I$ and $[B, \Rad(L)]=[S\cap I, \Rad(L)]\subseteq \Rad(L)\cap I=A$, which proves the result.
\end{proof}

According to \cite{Gr11}*{Chapter I, Section 3.11}, the direct product of two lattices $\mathcal{L}_1$ and $\mathcal{L}_2$ is the lattice on the set $\mathcal{L}_1\times \mathcal{L}_2$ with joint and meet defined componentwise.

\begin{corollary}\label{reduction-1} Let $L$ be a nonsolvable Lie algebra and $L^*=S\oplus \Rad(L^*)$ a faithful subalgebra of $L$ such that $L=\Ssoc(L)\oplus L^*$. The ideals of $L$ are of the form $B_s\oplus I^*$, where $I^*$ is an ideal of $L^*$ and $B_s\subseteq \Ssoc(L)$ is the sum of a finite number of simple ideals of $L$. In particular, $\mathcal{L}_{\id}(\Ssoc(L)\oplus L^*)$ is the direct product lattice $\mathcal{L}_{\id}(\Ssoc(L))\times \mathcal{L}_{\id}(L^*)$ and $\mathcal{L}_{\id}(L)$ is a finite lattice if and only if so is the lattice $\mathcal{L}_{\id}(L^*)$. \hfill $\square$
\end{corollary}

\begin{example} \emph{$(n,k)$-lattices $Q_n\times C_k (k\geq 2)$:} Let $L_{n,k}=S_n\oplus M_k$ be a Lie algebra, where $S_n$ is a semisimple ideal of $L$ with $n$ simple components, and $M_k$ is an ideal such that $\mathcal{L}_{\id}(M_k)=C_k$. According to \emph{\cite{Be92_1}}, $M_k$ is faithful, so $L^*_{n,k}=M_k$ and $\mathcal{L}_{\id}(L_{n,k})=\mathcal{L}_{\id}(S_n)\times \mathcal{L}_{\id}(M_k)$ is isomorphic to $Q_n\times C_k$. In Figure~\ref{fig:qnxck} we have drawn $\mathcal{L}_{\id}(L_{n,k})$ for $n=0,1,2$.
\end{example}

\begin{figure}
  \centering
  \includegraphics[width=0.6\textwidth]{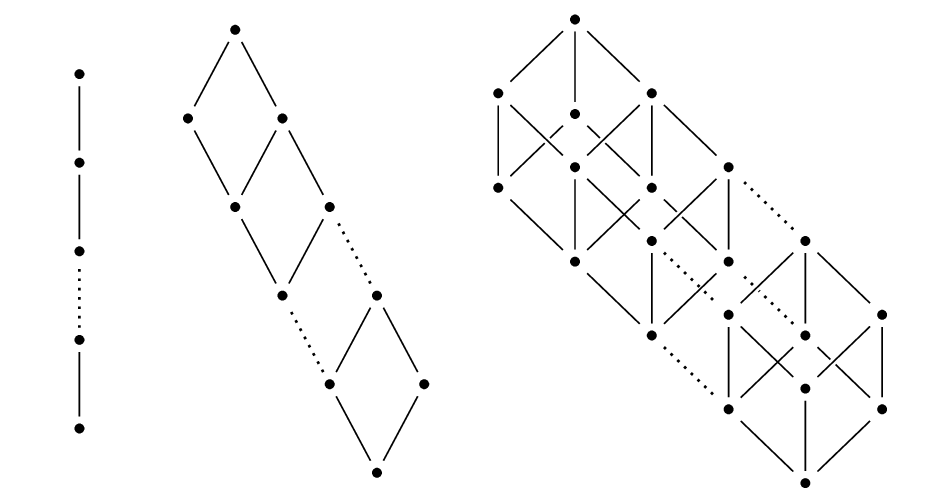}
  \caption{$Q_n\times C_k$ lattices for $n=0,1,2$.}
  \label{fig:qnxck}
\end{figure}

In Section 1, we have introduced the Jacobson radical of $L$, from now on $\mathcal{J}(L)$, as the intersection of all maximal ideals of $L$. We recall that:
\begin{equation}\label{eq:jacobsonradical}
  \Rad(L)^2\subseteq\mathcal{J}(L)=[L,\Rad(L)]=\Rad(L^2)=\Rad(L)\cap L^2\subseteq \Nil(L),
\end{equation}
where $\Rad(L)$ and $\Nil(L)$ are the solvable and nilpotent radicals of $L$. So, for any Levi decomposition $L=S\oplus \Rad(L)$, we get
\begin{equation}\label{eq:derivado}
  L^2=[L,L]=[S,S]+[L,\Rad(L)]=S\oplus \mathcal{J}(L).
\end{equation}
Hence,
\begin{itemize}
  \item[$\bullet$] $L$ is reductive iff $\mathcal{J}(L)=0$ iff $L=S\oplus \Z(L)$,
  \item[$\bullet$] $L$ is semisimple iff $\mathcal{J}(L)=0$ and $L^2=L\neq 0$,
  \item[$\bullet$] $L$ is solvable iff $L^2=\mathcal{J}(L)$ iff $S=0$. In this case, $L\neq L^2$ if $L\neq 0$.
\end{itemize}

A minimal ideal of $L$ is either simple or abelian. The socle of $L$, denoted by $\Soc(L)$, is the sum of all minimal ideals of $L$. The abelian socle of $L$, $\Asoc(L)$, is the sum of those minimal ideals that  are abelian, so $\Asoc(L)\subseteq \Nil(L)$ and $\Soc(L)=\Asoc(L)\oplus \Ssoc(L)$. Following~\cite{St70}*{Proposition 8}, any minimal ideal centralizes the nilradical, so $\Soc(L)\subseteq C_L(\Nil(L))$, equivalently, $[\Soc(L),\Nil(L)]=0$, and therefore $\Asoc(L)$ is contained in the centre of the nilradical, $\Z(\Nil(L))$. Remember that $\Ssoc(L)$ is contained in any Levi factor of $L$. The (adjoint) representation $\ad_{\Asoc(L)}\colon L\to \mathfrak{gl}(\Asoc(L))$ provides the Lie algebra of linear transformations $\Omega=\ad_{\Asoc(L)} L$ in $\Asoc(L)$. Since $\Omega$ is completely reducible, Lemma~\ref{completely-reducible} and its proof give us a clear structure of $\Omega$ as a Lie algebra of linear maps.

\begin{remark}\label{jacobsonreduced}
  For any decomposition $L=A\oplus B$, where $A$ and $B$ are ideals of $L$, if either $A$ is a sum of simple ideals or $A\subseteq \Z(L)$, then $\mathcal{J}(L)=\mathcal{J}(B)$.
\end{remark}

\begin{lemma}\label{basic-structure} Let $L$ be a nonzero Lie algebra, let $\mathcal{J}(L)$ be its Jacobson radical and let $S$ be a Levi factor of $L$ (the case $S=0$ is also possible). If $\mathcal{L}_{\id}(L)$ is finite, the following assertions hold:
  \begin{itemize}
    \item [\emph{(a)}] Either $L=L^2=S\oplus \mathcal{J}(L)$ with $S\neq 0$ and $\Rad(L)=\mathcal{J}(L)=\Nil(L)$ or $L=L^2\oplus \mathbb{F}a$ and $\Rad(L)=\mathcal{J}(L)\oplus \mathbb{F}a$. Moreover, $L^2=\mathcal{J}(L)=\Nil(L)$ in the case $S=0$ and $\mathcal{J}(L)\neq 0$.

    \item [\emph{(b)}] $\dim \Z(L)\leq 1$.

    \item [\emph{(c)}] $C_L(L^2)\nsubseteq L^2$ if and only if $L=L^2\oplus \mathbb{F}z$ and $C_L(L^2)=\Z(L)=\mathbb{F}z$. In this case, $L^2$ is a centreless Lie algebra, $\mathcal{L}_{\id}(L^2)$ is finite and $(L^2)^2=L^2$.

    \item [\emph{(d)}] $C_L(L^2)$ is an abelian ideal of $L$. Moreover, if $\Z_0(\Nil(L))$ is the trivial $\ad S$-module of $\Z(\Nil(L))$, then $\Z(L)\subseteq \Z_0(\Nil(L))\subseteq C_L(L^2)$. If $S=0$, $\Z_0(\Nil(L))=\Z(\Nil(L))$.

    \item [\emph{(e)}] Let $\Z_1(\Nil(L))$ be the sum of all nontrivial $\ad S$-modules of $\Z(\Nil(L))$, then $\Z(\Nil(L))=\Z_0(\Nil(L))\oplus \Z_1(\Nil(L))$ and $\Z_1(\Nil(L))\subseteq \mathcal{J}(L)$.
    \item [\emph{(f)}] Either $L=L^2\oplus \mathbb{F}z$ and $C_L(L^2)=\Z(L)=\mathbb{F}z=\Z_0(\Nil(L))$ or $\Z(\Nil(L))\subseteq \mathcal{J}(L)$.

    \item [\emph{(g)}] The ideals $I$ of $L$ such that $I\cap \mathcal{J}(L)=0$ are: the trivial ideal, any sum $I_s$ of simple ideals and, if $L=L^2\oplus \Z(L)$, we also have $\Z(L)=\mathbb{F}z$ and $I_s\oplus \mathbb{F}z$.


    \item [\emph{(h)}] If $L$ is a faithful and centreless Lie algebra, $I\cap \mathcal{J}(L)\neq 0$ for any nonzero ideal $I$ of $L$.

    \item [\emph{(i)}] $\Asoc(L)\subseteq \Z(\Nil(L))$ and either $\Asoc(L)=\Z(\Nil(L))=\Asoc(L)\cap \mathcal{J}(L)\oplus \mathbb{F}z$ and $\Z(L)=\mathbb{F}z=\Z_0(\Nil(L))$ or $\Z(\Nil(L))\subseteq \mathcal{J}(L)$.

    \item [\emph{(j)}] Any homomorphic image of $L$ has a finite number of ideals. So, for any ideal $I$ of $L$, the quotient Lie algebra $L/I$ satisfies \emph{(a)--(i)}.

  \end{itemize}
\end{lemma}

\begin{proof}
  The final assertion (j) is clear. Since any subspace contained in the centre of $L$ is an ideal, the dimension of $\Z(L)$ is at most $1$, thus (b) holds. On the other hand, the quotient $L/L^2$ is abelian, so $\Z(L/L^2)=L/L^2$. Since $\mathcal{L}_{\id}(L/L^2)$ is finite, the previous reasoning on the centre yields item (a). We point out that $L$ nilpotent and $L^2\neq 0$ imply $\dim L/L^2\geq 2$, so the final assertion in this item is clear.
  Now, we will prove the equivalence in item~(c). The converse is straightforward. Assume then $C_L(L^2)\nsubseteq L^2$, thus $L=L^2\oplus \mathbb{F}a$ for any $a\in C_L(L^2)\setminus L^2$ by (a). Since $[a,L]=0$, we get $\Z(L)=\mathbb{F}a$. For any $y\in C_L(L^2)$, we have that $[y,L]=[y,L^2]=0$ and therefore $C_L(L^2)=\Z(L)$. The final assertion in (c) is clear. First half in item~(d) is obvious in the case $C_L(L^2)\subseteq L^2$ and it follows from (c) if $C_L(L^2)\nsubseteq L^2$. The second half for $S\neq0$ is a consequence of $L^2=S\oplus \mathcal{J}(L)$ and $\mathcal{J}(L)\subseteq \Nil(L)$. If $S=0$ it is straightforward. Since $\Z(\Nil(L))$ is an $\ad S$-module, the decomposition of $\Z(\Nil(L))=Z_0\oplus Z_1$ in item~(e) follows from Weyl's Theorem. Moreover, $[S,I]=I$ for any $I$ $\ad S$-irreducible and nontrivial module, thus $Z_1=[S,Z_1]\subseteq[S,\Rad(L)] \subseteq\mathcal{J}(L)$.

  Next we will prove (f). In the case $\mathcal{J}(L)=0$, $L$ is reductive with centre of dimension up to $1$ and the result follows. Assume then $\mathcal{J}(L)\neq 0$. If $L$ is a solvable algebra, $L^2=\mathcal{J}(L)\neq 0$ and $\Z(\Nil(L))\subseteq \mathcal{J}(L)=\Nil(L)$ by (a). Otherwise $S\neq 0$ and consider the $\ad S$-module decomposition $Z=\Z(Nil(L))=Z_0\oplus Z_1$ in item~(e). Since$$Z\cap \mathcal{J}(L)=(Z_0\oplus Z_1)\cap \mathcal{J}(L)=Z_0\cap\mathcal{J}(L)\oplus Z_1,$$if $Z\nsubseteq \mathcal{J}(L)$ there exists $x\in Z_0\setminus \mathcal{J}(L)$. By applying~(d) we get $Z_0\subseteq C_L(L^2)\subseteq \Rad(L)$, thus $C_L(L^2)\nsubseteq\mathcal{J}(L)=\Rad(L)\cap L^2$. Therefore $C_L(L^2)\nsubseteq L^2$, and the result follows from items~(c) and (d).

  Now, let $I=I\cap S\oplus I\cap \Rad(L)$ be an ideal such that $I\cap \mathcal{J}(L)=0$. Since $[I\cap S,\Rad(L)]\subseteq[I,\Rad(L)]\subseteq I\cap \mathcal{J}(L)=0$, the subspace $I\cap S$ is either null or a semisimple ideal of $L$. Then, $I\cap S$ decomposes as a direct sum of simple ideals of $L$ if $I\cap S\neq 0$. In the same way, $[I\cap \Rad(L),L]\subseteq I\cap \mathcal{J}(L)=0$ and therefore $I\cap \Rad(L)\subseteq \Z(L)$. Then, either $I\cap \Rad(L)=0$ or $I\cap \Rad(L)=\Z(L)=\mathbb{F}z\nsubseteq \mathcal{J}(L)$ and $L=L^2\oplus \mathbb{F}z$ by applying (f). Hence item~(g) holds and (h) follows as a corollary.

  Finally, to prove (i) we recall that $\Asoc(L)\subseteq \Z(\Nil(L))=Z_0\oplus Z_1$. Assume $\Z(\Nil(L))$ is not contained in $\mathcal{J}(L)$. By (f), $L=S\oplus \mathcal{J}(L)\oplus \mathbb{F}z$ and $Z_0=\Z(L)=\mathbb{F}z$. Then, using $[\mathcal{J}(L), Z_1]=0$ we get that each $\ad S$-irreducible module of $Z_1$ is a minimal abelian ideal. Thus $\Asoc(L)=\Z(\Nil(L))$ and (i) holds.
\end{proof}

\begin{lemma}\label{reduction-2}
  Let $L=M\oplus \mathbb{F}z$, where $M$ is an ideal of $L$ such that $M^2=M$ and $z\in \Z(L)$. Then, we have that $L^2=M^2=S\oplus \mathcal{J}(L)$, where $S$ is a nonzero Levi factor of $M$, $\mathcal{J}(L)=\mathcal{J}(M)$ and we also have that $\Rad(L)=\mathbb{F}z\oplus\mathcal{J}(L)$. Moreover, the ideals of $L$ are of one of the following types:
  \begin{itemize}
    \item[\emph{(i)}]Any ideal $I_M$ of $M$,
    \item[\emph{(ii)}]$I_M\oplus \mathbb{F}z$ where $I_M$ is an ideal of $M$,
    \item[] or the trivial $\ad S$-module $\mathcal{J}_0(M)$ of $\mathcal{J}(M)$ is nonzero and
    \item[\emph{(iii)}]$\mathbf{I}(I_M,x)= I_M\oplus\mathbb{F}(x+z)$ where $I_M$ is an ideal of $M$, $x\in \mathcal{J}_0(M)-I_M$ and $[x,\mathcal{J}(M)]\subseteq I_M$.
  \end{itemize}
  For a fixed ideal $\mathbf{I}(I_M,x)$ of type \emph{(iii)}, the subspaces $\mathbf{I}(I_M, \alpha x)$ are ideals for all $\alpha \in \mathbb{F}$ and
  $\mathbf{I}(I_M,\alpha x)=\mathbf{I}(I_M,\beta x)$ if and only if $\alpha=\beta$.
\end{lemma}

\begin{proof}
  Let $I$ be any ideal of $L$. In the case that $z\in I$, we get \emph{(ii)}. Assume then $z\notin I$ and consider the projection $p$ of the ideal $I$ on $\mathbb{F}z$. Note that $p(I)=0$ is equivalent to the fact that $I$ is an ideal of $M$, so we arrive at (i). Otherwise, $\ker p=I\cap M$ and $p(I)=\mathbb{F}z$. Hence, there exists a nonzero element $a\in I$ such that $a=x+z$ with $x\in M\setminus I$. Then, we can decompose $I=I\cap M \oplus \mathbb{F}a$, $[a,L]=[a,M]\subseteq I\cap M$. Since $I\cap M$ and $I$ are $\ad S$-modules, we can assume w.l.o.g. that $a\in C_L(S)\subseteq \Rad(L)=\mathcal{J}(L)\oplus \mathbb{F}z$. Hence, $x\in \mathcal{J}(L)$ and $[S,x]=0$ and therefore the ideal $I$ is as in (iii). The first assertion of $L^2$ and the final part about ideals of type (iii) are straightforward.
\end{proof}

In the sequel, we will work on faithful Lie algebras ($\Ssoc(L)=0$) with a finite number of ideals and nontrivial Jacobson radical. The nonfaithful case is covered by Corollary~\ref{reduction-1} and the case of null Jacobson radical by Corollary~\ref{cor-booleano}. We also consider the intersection of all maximal subalgebras of $L$, denoted by $\varphi(L)$, which is known as the \emph{Frattini subalgebra} of $L$. Following~\cite{To73}*{Theorem~6.5 and Corollary~7.7}, in characteristic zero, $\varphi (L)$ is a characteristic ideal and
\begin{equation}\label{eq:frattini}
  \Nil(L)^2\subseteq \varphi(L)\subseteq \Rad(L)^2.
\end{equation}
\begin{proposition}\label{general-est1}
  Let $L$ be a faithful Lie algebra with nonzero Jacobson radical $\mathcal{J}(L)$, and let $S$ be a Levi factor of $L$. If $\mathcal{L}_{\id}(L)$ is finite, then $\dim \Z(L)\leq 1$ and $L$ satisfies one of the following decomposition features:
  \begin{itemize}

    \item [\emph{(I)}]$L=L^2=S\oplus \mathcal{J}(L)$ and $S\neq 0$, $\Rad(L)=\Nil(L)=\mathcal{J}(L)$ and the representation induced on the quotient $\mathcal{J}(L)/\mathcal{J}(L)^2$ by the adjoint representation $\ad_{\mathcal{J}(L)} S$ has non trivial and non isomorphic modules. Moreover:
          \begin{itemize}
            \item[$\circ$] $\Asoc(L)=\Z(\Nil(L))=I_1\oplus \dots \oplus I_r$, where $I_j$ are irreducible and non isomorphic $\ad_{\Z(\Nil(L))} S$-modules and,
            \item[$\circ$] $\Z(L)=Z_0(\Nil(L))=C_L(L^2)$.

          \end{itemize}
    \item [\emph{(II)}]$L=M\oplus \mathbb{F}z$, where $\Z(L)=\mathbb{F}z=C_L(L^2)$ and $M=S\oplus \mathcal{J}(L)=S\oplus \mathcal{J}(M)$ is an algebra as in item \emph{(I)} such that $\mathcal{J}_0(M)=0$. In this case, $\mathcal{L}_{\id}(L)=\mathcal{L}_{\id}(M)\times \mathcal{L}_{\id}(\mathbb{F}z)$.

    \item [\emph{(III)}] $C_L(L^2)\subseteq \mathcal{J}_0(L)$ and $L=L^2\oplus \mathbb{F}a=S\oplus \mathcal{J}(L)\oplus\mathbb{F}a$ with $S\neq 0$, $\Rad(L)=\mathcal{J}(L)\oplus \mathbb{F}a$ and $[S,a]=0$. In this case, $\Rad(L)^2=\mathcal{J}(L)^2+[\mathcal{J}(L),a]$, $[\mathcal{J}(L),a]\neq 0$ and $L/\Rad(L)^2$ is as in \emph{(II)} if $\mathcal{J}(L)\neq \Rad(L)^2$. Moreover:
          \begin{itemize}
            \item[$\circ$] $\Asoc(L)\subseteq \Z(\Nil(L))\subseteq \mathcal{J}(L)$.
            \item[$\circ$] The minimal abelian ideals of $L$ are just the irreducible $\ad (S\oplus \mathbb{F}a)$-subspaces of $\Z(\Nil(L))$. In particular, $\Asoc(L)=I_1\oplus \dots \oplus I_r$ where $I_1,\dots ,I_r$ are                   irreducible and non isomorphic $\ad (S\oplus \mathbb{F}a)$-subspaces.
            \item[$\circ$] The map $\ad_{\Asoc(L)}a$ is semisimple and the minimum polynomial of $\ad_{I_j} a\colon I_j\to I_j$, is irreducible. Moreover, if $\ad_{I_j} a=\lambda_j\cdot \id_{I_j}$, then $I_j$ is an irreducible $\ad S$-subspace.
            \item[$\circ$] $\Z(L)=C_{Z_0(\Nil(L))}(a)=C_L(L^2)\cap \ker \ad a$.
            \item[$\circ$] If $\Nil(L)=\mathcal{J}(L)$, then $Z_0(\Nil(L))=C_L(L^2)$ and ${\ad}_{\mathcal{J}(L)/\mathcal{J}(L)^2} a$ is not a nilpotent map.
            \item[$\circ$] If $\Nil(L)=\mathcal{J}(L)\oplus \mathbb{F}a=\Rad(L)$, then $\ad_L a$ is a nilpotent map and $\mathcal{J}_0(L)\subseteq \varphi(L)$. In this case, $\Asoc(L)=\Z(\Nil(L))$, $Z_0(\Nil(L))=\Z(L)$ and the irreducible $\ad S$-modules of $Z_1(\Nil(L))$ are the noncentral minimal ideals of $L$.
          \end{itemize}
    \item [\emph{(IV)}] $C_L(L^2)\subseteq \mathcal{J}(L)$ and $L=\mathcal{J}(L)\oplus \mathbb{F}a$ is a solvable Lie algebra with $\Nil(L)=\mathcal{J}(L)=L^2$. In this case, the map induced on the quotient $\mathcal{J}(L)/\mathcal{J}(L)^2$ by $\ad_{\mathcal{J}(L)} a$ is one-to-one. Moreover:
          \begin{itemize}
            \item[$\circ$] For any vector space $T$ such that $\mathcal{J}(L)=\mathcal{J}(L)^2\oplus T$, we have $\mathcal{J}(L)=\mathcal{J}(L)^2\oplus [T,a]$.
            \item[$\circ$] The minimal abelian ideals of $L$ are just the irreducible $\ad a$-subspaces of $\Z(\Nil(L))$.
            \item[$\circ$] $\Asoc(L)=Z_{\pi_1}^*\oplus\dots \oplus Z_{\pi_r}^*$, where $\pi_j=\pi_j(t)$ are the distinct irreducible polynomials of the minimum polynomial of $\ad a$ and $Z_{\pi_j}^*=\{x\in \Z(\Nil(L)):\pi_j(\ad a)(x)=0\}$. In the case $Z_{\pi_j}^*\neq 0$, $\ad a$ acts cyclically on $Z_{\pi_j}^*$, so $\dim Z_{\pi_j}^*=\degr \pi_j$.
            \item[$\circ$] $\Z(\Nil(L))=C_L(L^2)$.
            \item[$\circ$] $\Z(L)=C_{\Z(\Nil(L))}(a)=Z^*_\pi$ where $\pi(t)=t$.

          \end{itemize}
  \end{itemize}
  In addition we have:
  \begin{itemize}
    \item[$\circ$] $\mathcal{J}(L)^2\subseteq\varphi(L)\subseteq \mathcal{J}(L)^2+[\mathcal{J}(L),a]\subseteq \mathcal{J}(L)$ and $\varphi(L)$ is exactly $\mathcal{J}(L)^2$ for any Lie algebra different from items \emph{(III)} and \emph{(IV)}. If $L$ is as in item \emph{(III)} and $\Rad(L)=\Nil(L)$, then $\varphi(L)=\mathcal{J}(L)^2+[\mathcal{J}(L),a]\neq \mathcal{J}(L)$.

    \item[$\circ$] If $0\neq \Z(L)=\mathbb{F}z$, then either $L$ is as in item \emph{(II)} or $z\in \varphi(L)$.

  \end{itemize}

\end{proposition}

\begin{proof}
  From item~(a) in Lemma~\ref{basic-structure}, either $L=L^2$ or $L\neq L^2$. The first possibility gives us $L=L^2=S\oplus \mathcal{J}(L)$, where $\mathcal{J}(L)\neq 0\neq S$ and $\Rad(L)=\Nil(L)=\mathcal{J}(L)$. We also note that $\mathcal{J}(L)^2$ and $\Z(\Nil(L))$ are $\ad S$-modules via the adjoint representation $\ad_{\mathcal{J}(L)}S$. Since $L$ is not semisimple, there is a nonzero $S$-module $M$ such that $\mathcal{J}(L)=\mathcal{J}(L)^2\oplus M$. Then, from $\mathcal{J}(L)=[L,\Rad(L)]$ we get:$$
    \mathcal{J}(L)=[S\oplus\mathcal{J}(L),\mathcal{J}(L)]=[S,\mathcal{J}(L)]+\mathcal{J}(L)^2=\mathcal{J}(L)^2\oplus[S,M].$$Hence $[S,M]=M$, which implies that $M$ has no trivial submodules.
  Consider now $Z=\Z(\Nil(L))$ and note that $C_L(L^2)=C_L(L)=\Z(L)$. Then, the identity $\Z(L)=Z_0=C_L(L^2)$ follows from~(d) in Lemma~\ref{basic-structure}. On the other hand, $Z$ is an $\ad S$-module and the irreducible $\ad S$-submodules of $Z$ are the minimal ideals of $L$, so $Z=\Asoc(L)$ because of~(i) in Lemma~\ref{basic-structure}. In this way, the set $\Sub_\Omega Z$ of subspaces of $Z$ invariant under $\Omega=\ad_Z S$ is a complemented and finite sublattice of $\mathcal{L}_{\id}(L)$. Hence, there are no $\Omega$-isomorphic elements according to Lemma~\ref{completely-reducible}. The assertion on the decomposition of $\mathcal{J}(L)/\mathcal{J}(L)^2$ follows by using the previous argument on the quotient Lie algebra $L_1=L/\mathcal{J}(L)^2$ taking into account that $\Asoc(L_1)=\Z(\Nil(L_1))=\mathcal{J}(L)/\mathcal{J}(L)^2$. Therefore, the assumption $L=L^2$ gives $L$ of type~(I).

  From now on we will describe the case $L\neq L^2 $. If $C_L(L^2)\nsubseteq L^2$, by item~(c) of Lemma~\ref{basic-structure}, we get $L=M\oplus \mathbb{F}z$, $M=L^2$, $M^2=(L^2)^2=M$ and $\Z(L)=\mathbb{F}z$. Moreover, $\mathcal{J}(L)=\mathcal{J}(M)$ by Remark~\ref{jacobsonreduced}. Denote by $\mathcal{J}_0=\mathcal{J}_0(M)=\mathcal{J}_0(L)$ the trivial $\ad S$-submodule of $\mathcal{J}(L)$ and let $x$ be an element of $\mathcal{J}_0$. If there is some $k_0\geq 1$ such that $x\in \mathcal{J}^{k_0}- \mathcal{J}^{k_0+1}$, since $$[x,\mathcal{J}]\subseteq [\mathcal{J},\mathcal{J}^{k_0}]=\mathcal{J}^{k_0+1},$$
  the set $\{\mathbb{F}(\alpha x+z)\oplus \mathcal{J}^{k_0+1}: \alpha\in \mathbb{F}\}$ is an infinite family of ideals of $L$ according to Lemma~\ref{reduction-2}, a contradiction. Hence, $x\in \mathcal{J}^{k}$ for all $k\geq 1$ and therefore $x=0$ because of $\mathcal{J}$ is nilpotent. So $\mathcal{J}_0=0$ and $\mathcal{L}_{\id}(L)$ is a direct product of a $1$-chain and the $\mathcal{L}_{\id}(M)$ according to Lemma~\ref{reduction-2}. Therefore $L$ is of type~(II).

  Otherwise $C_L(L^2)\subseteq L^2$ and $L=L^2\oplus \mathbb{F}a$ because of (a) in Lemma~\ref{basic-structure}. Then, $C_L(L^2) \subseteq L^2\cap \Rad(L)=\mathcal{J}(L)$ and $\Z(\Nil(L)))\subseteq \mathcal{J}(L)$ by applying items (d) and (e) in Lemma~\ref{basic-structure}. Assume firstly that $L$ is solvable, so $L^2=\mathcal{J}(L)\subseteq \Nil(L)$. Let $T$ be a vector space such that $\mathcal{J}(L)=\mathcal{J}(L)^2\oplus T$. Then$$
    \mathcal{J}(L)=L^2=[\mathcal{J}(L)\oplus \mathbb{F}a,\mathcal{J}(L)\oplus \mathbb{F}a]=\mathcal{J}(L)^2+[T,a].$$
  The linear map $\ad a\colon T\to \mathcal{J}(L)$ sets up the equation:
  $$\dim T=\dim [T,a]+\dim C_T(a),$$
  here $C_T(a)=\{t\in T: [t,a]=0\}$. From $\mathcal{J}(L)=\mathcal{J}(L)^2+[T,a]=\mathcal{J}(L)^2\oplus T$, we get $[T,a]\cap \mathcal{J}(L)^2=0$ and $\dim T=\dim [T,a]$. Then $\mathcal{J}(L)=\mathcal{J}(L)^2\oplus [T,a]$ and $\ad_{\mathcal{J}(L)} a$ induces a one-to-one map on $\mathcal{J}(L)/\mathcal{J}(L)^2$. This also implies that $\mathcal{J}(L)=\Nil(L)$ because of $\ad a$ is not nilpotent. The identity $\Z(L)=C_{\Z(\Nil(L))}(a)$ is clear and $\Z(\Nil(L))=C_L(L^2)$ follows from the fact that $C_L(L^2)$ is an abelian ideal according to (d) in Lemma~\ref{basic-structure}. By item (i) of the latter Lemma, the minimal ideals of $L$ are just the $\ad a$-irreducible subspaces of $\Z(\Nil(L))$. Hence $L$ is as in item~(IV), the features of the decomposition of $A=\Asoc(L)$ follow from $\Omega =\mathbb{F}\cdot \ad_A a$ and Lemma~\ref{completely-reducible}.


  We assume finally $C_L(L^2) \subseteq L^2$, thus $C_L(L^2) \subseteq \mathcal{J}(L)$ and $\Z(\Nil(L))\subseteq \mathcal{J}(L)$, and $L=L^2\oplus \mathbb{F}a$ not solvable, so $S\neq 0$. We point out that $C_L(L^2)$ is a trivial $S$-module because of $S\subseteq L^2$. Hence $C_L(L^2)\subseteq \mathcal{J}_0(L)$. Since $L^2=S\oplus \mathcal{J}(L)$ is an $S$-module we can assume w.l.o.g. $a\in C_L(S)$. Then $[S,a]=0$ and $a\notin C_L(L^2)$ yields $[\mathcal{J}(L),a]\neq 0$. Now $\Rad(L)^2=[\mathcal{J}(L)\oplus \mathbb{F}a,\mathcal{J}(L)\oplus \mathbb{F}a]=\mathcal{J}(L)^2+ [\mathcal{J}(L),a]$ and the quotient Lie algebra $M=L/\Rad(L)^2$ is as in item~(II) if $\mathcal{J}(L)\neq \Rad(L)^2$ because of $a+\Rad(L)^2\in C_M(M^2)\nsubseteq M^2$.

  We denote $Z=\Z(\Nil(L))=Z_0\oplus Z_1$ and observe that $Z$ is $\ad (S\oplus \mathbb{F}a)$-invariant. Since $A=\Asoc(L)\subseteq Z$, the minimal abelian ideals of $L$ are exactly the irreducible $\ad (S\oplus \mathbb{F}a)$-subspaces of $Z$. Then, $A=I_1\oplus \dots \oplus I_r$ where each $I_j$ is $\ad (S\oplus \mathbb{F}a)$-irreducible. By setting $\Omega =\ad_A(S\oplus \mathbb{F}a)$, the different $I_j$ are not $\Omega$-isomorphic according to Lemma~\ref{completely-reducible} and $\ad_{A}a$ is a semisimple map. Let $I\in \{I_1,\dots,I_r\}$ be and define $I_{\pi}^*=\{x\in I:\pi(\ad a)(x)=0\}$. Consider now $I=I_{\pi_1}^*\oplus\dots \oplus I_{\pi_r}^*$ the decomposition of $I$ relative to the different irreducible polynomials $\pi_j$ of the minimum polynomial of $\ad a$ (decomposition into primary components of $I$). Since $\ad a$ commutes with $\ad s$ for any $s\in S$, every subspace $I_{\pi_j}^*$ is $\ad S$-invariant. Therefore $I_{\pi_j}^*$ is an ideal of $L$ that decomposes into irreducible $\ad S$-modules. Since $I$ is minimal, there is a unique $\pi_k$ such that $I_{\pi_k}^*\neq 0$, so $I=I_{\pi_k}^*$. In the case $\pi_k(t)=t-\lambda$, the ideal $I$ is $\ad S$-irreducible.

  From (d) in Lemma~\ref{basic-structure}, $\Z(L)\subseteq Z_0\subseteq C_L(L^2)$. Since $[C_L(L^2),L^2]=0$, the equality $\Z(L)=C_L(L^2)\cap \ker \ad a$ is clear. If $\mathcal{J(L)}=\Nil(L)$, the map $\ad a$ is not nilpotent in $\Nil(L)/\Nil(L)^2$ and  $C_L(L^2)\subseteq Z_0$ because of $C_L(L^2)\subseteq \mathcal{J}_0(L)$ and $[C_L(L^2),\mathcal{J}(L)]=0$. Otherwise, $\Nil(L)=\mathcal{J}(L)\oplus \mathbb{F}a=\Rad (L)$, $\ad a$ is nilpotent and $\Z=Z_0\oplus Z_1\subseteq \ker \ad a$. So, $Z_0=\Z(L)$ and any irreducible $\ad S$-module of $Z_1$ is a minimal ideal of $L$. This implies that $\Asoc(L)=\Z(\Nil(L))$. The assumption $\Nil(L)=\Rad(L)$ yields $\varphi(L)=\Rad(L)^2$ according to~\cite{To73}*{Corollary 7.8}. Note that $\Rad(L)^2\neq \mathcal{J}(L)$ because of $\Rad(L)^2\neq 0$ and  $\dim \Rad(L)/\mathcal{J}(L)=1$. Thus $L/\varphi(L)$ is as in item~(II), so $\mathcal{J}_0(L/\varphi(L))=0$ which implies $\mathcal{J}_0(L)\subseteq \varphi(L)$. Therefore $L$ is of type (III).

  It remains for us to prove the two final assertions about $\varphi(L)$. The first one follows from~\cite{To73}*{Theorem 4.8, and Corollaries~7.7 and~7.8}. We then assume that $0\neq \Z(L)=\mathbb{F}z$ and $L$ is not as in (II). For any Lie algebra of types (I), (III) or (IV) $\mathbb{F}z\subseteq \Nil(L)\cap C_L(L^2)\subseteq L^2$ and then $z\in \varphi(L)$ according to~\cite{Ma67}.
\end{proof}

\section {Lie algebras with finite lattice of ideals and trivial Frattini subalgebra}

According to \cite{To73}, any Lie algebra with trivial Frattini subalgebra splits over its abelian socle; the converse also holds. Among the Lie algebras with a finite number of ideals, those with trivial Frattini subalgebra will be described through several equivalent conditions in this section.

\begin{lemma}\label{free-1}
  Let $L$ be a Lie algebra such that $\mathcal{L}_{\id}(L)$ is finite. Then, the following conditions are equivalent:
  \begin{itemize}
    \item [\emph{(a)}]The Frattini subalgebra of $L$ is trivial.

    \item [\emph{(b)}]$\Nil(L)=\Asoc(L)$.

    \item [\emph{(c)}]$\mathcal{J}(L)\subseteq \Asoc(L)$ and $\mathcal{J}(L)\cap \Z(L)=0$.

    \item [\emph{(d)}]Either $\Asoc(L)=\mathcal{J}(L)$ and $L$ is centreless or $\Asoc(L)=\mathcal{J}(L)\oplus \mathbb{F}z$ and $\Z(L)=\mathbb{F}z$.

  \end{itemize}
\end{lemma}

\begin{proof}
  If $\varphi(L)=0$, the abelian socle and the nilradical are equals (see~\cite{To73}*{Theorem 7.3}) and therefore (b) holds. Now, if we assume $\Nil(L)=\Asoc(L)$, clearly $\mathcal{J}(L)\subseteq \Asoc(L)$. Moreover, $\mathcal{J}(L)\cap \Z(L)\subseteq L^2\cap \Z(L)\subseteq \varphi(L)=0$ according to~\cite{Ma67}. Hence (c) follows. The assertion (i) of Lemma~\ref{basic-structure} gives us (d) from (c). Finally, from the structure results on Lie algebras of a finite number of ideals stablished in Proposition~\ref{general-est1}, it is easily checked that there is a subalgebra $M$ such that $M\oplus \Asoc(L)$ if (d) is assumed. So $L$ splits over its abelian socle and therefore $\varphi(L)=0$ by applying~\cite{To73}*{Theorem 7.3}.
\end{proof}

\begin{theorem}\label{free-2}
  Any Lie algebra $L$ with $\mathcal{L}_{\id}(L)$ finite and trivial Frattini subalgebra is of one of the following types:
  \begin{itemize}
    \item [\emph{\textbf{ A.} }] Either $0$ or $\mathbb{F}z$.

    \item [\emph{\textbf{ B-I.}}] $L=S\oplus A$, where $S$ is a nonzero semisimple algebra, $\Ssoc(L)=0$ and $A$ is a nonzero abelian ideal that decomposes as a direct sum of $k\geq 1$ non trivial and non isomorphic irreducible $S$-modules. In this case, $L$ is centreless, $A=\Asoc(L)=\Rad(L)$ and the number of solvable ideals of $L$ is $2^k$. Moreover the number of ideals of $L$, $\mathbf{n}_{\id}(L)$, satisfies:$$
            2^l+2^k-1\leq \mathbf{n}_{\id}(L)\leq 2^l\cdot 2^k+3-(2^l+2^k),
          $$where $l\geq 1$ is the number of simple components of $S$.

    \item [\emph{\textbf{ B-II.}}] $L=M \oplus \mathbb{F}z$, $[z,M]=0$, $\Ssoc(L)=0$ and $M$ as in item \emph{\textbf{B-I}}. In this case, $\mathcal{L}_{\id}(L)= \mathcal{L}_{\id}(M)\times\mathcal{L}_{\id}(\mathbb{F}z)$ and $\mathbf{n}_{\id}(L)=2\cdot \mathbf{n}_{\id}(M)$.

    \item [\emph{\textbf{ B-III.}}]$L=S \oplus A\oplus\mathbb{F}a$ is a centreless Lie algebra, $S$ is a nonzero semisimple subalgebra, $\Ssoc(L)=0$, $[S,a]=0$ and $A$ is a nonzero abelian ideal. We also have that $\ad_A a$ is a semisimple map with minimal polynomial $p_m(t)=\pi_1(t)\cdot \dots \cdot \pi_r(t)\neq t$, where $\pi_j$ are distinct irreducible polynomials. Moreover, the subspaces $A_0=\{x\in A: [S,x]=0\}$ and the sum of all non trivial $\ad S$-modules of $A$, denoted by $A_1$, are $\ad_A a$-invariant. In addition, $A_1\neq 0$ and the decomposition into primary components $(A_i)_{\pi_j}=\{x\in A_i:\pi_j(\ad a)(x)=0\}$,$$A=A_0\oplus A_1=(A_0)_{\pi_1}\oplus \dots\oplus (A_0)_{\pi_r}\oplus (A_1)_{\pi_1}\oplus \dots\oplus (A_1)_{\pi_r},$$satisfies:
          \begin{itemize}
            \item[$\circ$] Each nonzero component $(A_0)_{\pi_j}$ is a minimal ideal and $\dim (A_0)_{\pi_j}=\deg \pi_j$;
            \item[$\circ$] $(A_0)_t=0$ and $(A_1)_t$ decomposes as $k_0\geq 0$ irreducible and non isomorphic $\ad S$-modules;
            \item[$\circ$] each nonzero component $(A_1)_{\pi_j}$ decomposes into $k_j\geq 1$ irreducible $\ad (S\oplus \mathbb{F}a)$-subspaces which are non isomorphic.

          \end{itemize}
          In this case, $A_0$ decomposes into $0\le s\le r$ minimal ideals, $A_1$ decomposes into $k=k_0+k_{i_1}+\dots +k_{i_j}\geq 1$ minimal ideals and $L$ has exactly $2^{s+k}$ abelian ideals. Moreover:$$
            2^{l+1}+2^{s+k}+2^{k_0}-2\leq \mathbf{n}_{\id}(L)\leq (2^l-1)\cdot(2^{s+k}+2^{k_0}+2^s-1)+2,
          $$where $l\geq 1$ is the number of simple components of $S$.

    \item [\emph{\textbf{ C.}}] $L=A\oplus \mathbb{F}a$, where $A$ is a nonzero abelian ideal, $\ad_A a$ is a bijective and semisimple map with minimum polynomial equal to the characteristic one. In particular,$$A=(A)_{\pi_1}\oplus \dots\oplus (A)_{\pi_r},$$where $\pi_j(t)\neq t$ are the distinct irreducible polynomials of the minimum polynomial of $\ad_A a$ and $\dim (A)_{\pi_j}=\deg \pi_j$. In this case, $\mathbf{n}_{\id}(L)=2^r+1$.

    \item [\emph{\textbf{ D.}}] $L=M_s\oplus L^*$ is a direct sum as idels of a nonzero semisimple  Lie algebra $M_s$ and any Lie algebra $L^*$ as in items \emph{\textbf{A}, \textbf{B-I}, \textbf{B-II}, \textbf{B-III}} or \emph{\textbf{C}}. Hence, $\mathcal{L}_{\id}(L)=\mathcal{L}_{\id}(M_s)\times \mathcal{L}_{\id}(L^*)$ and $n_{\id}(L)=2^p\cdot n_{\id}(L^*)$, where $p\geq 1$ the number of simple components of $M_s$.

  \end{itemize}
\end{theorem}

\begin{proof}
  If $\mathcal{J}(L)= 0$, by Remark~\ref{equivalence reductive} we have $\Nil(L)=Z(L)=\Asoc(L)=\Rad(L)$ of dimension up to $1$ because of $\mathcal{L}_{\id}(L)$. Then, $L$ splits over its abelian socle and therefore $\varphi(L)=0$, and $L$ is as in item \textbf{A} if $\Ssoc(L)=0$. Assume then $\varphi(L)=0$, $\mathcal{J}(L)\neq 0$ and $\Ssoc(L)=0$. According to Lemma~\ref{free-1}, $\mathcal{J}(L)\subseteq \Asoc(L)$ and $\mathcal{J}(L)\cap \Z(L)=0$. Then, Proposition~\ref{general-est1} gives us the following possibilities:
  \begin{itemize}

    \item[$\bullet$] $L=S\oplus \mathcal{J}(L)$ is as in item~(I). Then $L$ is centreless and $A=\Asoc(L)=\mathcal{J}(L)$. In fact, the minimal abelian ideals are the irreducible $\ad S$-modules of $\mathcal{J}(L)$, which are non trivial and non isomorphic. Hence $L$ is as in item \textbf{B-I}. The lower bound on the number of ideals is clear. For the upper bound, we use Lemma~\ref{lema-reduction-1}. The ideals of $L$ are of the form $S'\oplus I$, where $I\subseteq \mathcal{J}(L)$ is a sum of minimal ideals and $S'$ is a semisimple subalgebra of $S$ such that $[S',\mathcal{J}(L)]\subseteq I$. So, at most there are $2^l\cdot 2^k$ ideals. Now, $I=0$ implies $S'=0$ because $\Ssoc(L)=0$, and if we take $S'=S$, then $I=\mathcal{J}(L)$ is the only possibility.

    \item[$\bullet$] $L$ is as in item~(II). Then $\Asoc(L)=\mathcal{J}(L)\oplus \mathbb{F}z$ and $L$ is as item \textbf{B-II}.

    \item[$\bullet$] $L$ is as in item (III). In this case $\Z(L)=0$ and $A=\Asoc(L)=\mathcal{J}(L)=\Nil(L)=\Z(\Nil(L))$. So $L=S\oplus A\oplus \mathbb{F}a$, $[S,a]=0$ and $\ad_A a$ is a semisimple and non nilpotent map. Let $p_m(t)=\pi_1(t)\cdot \dots \cdot \pi_r(t)\neq t$ be the decomposition of the minimum polynomial of $\ad_A a$ into distinct irreducible $\pi_j$, $A_0$ the trivial $\ad S$-module of $A$ and $A_1$ the sum of all nontrivial $\ad S$-modules. The commutativity of $\ad_A a$ and $\ad_A s$ for any $s\in S$ implies that $A_0=\mathcal{J}_0(L)$ and $A_1=\mathcal{J}_1(L)$ are $\ad_A a$-invariant. So we can consider the primary decomposition, $(A_i)_{\pi_j}=\{x\in A_i:\pi_j(\ad a)(x)=0\}$:
          \begin{itemize}
            \item[$\circ$]$A_0=(A_0)_{\pi_1}\oplus \dots\oplus (A_0)_{\pi_r}$, and
            \item[$\circ$]$A_1=(A_1)_{\pi_1}\oplus \dots\oplus (A_1)_{\pi_r}$.
          \end{itemize}

    \item[] The ideals inside $A_0$ are just the $\ad a$-invariant subspaces of $A_0$, so $\ad_A a$ acts cyclically on $(A_0)_{\pi_j}$, and therefore $\dim (A_0)_{\pi_j}=\deg \pi_j$. If $\ad_A a$ is not one-to one, we can assume w.l.o.g. $\pi_1=t$. In this case, $(A_0)_{\pi_1}=\Z(L)=0$ and $(A_1)_{\pi_1}$ is an ideal that decomposes as a finite number of irreducible and non isomorphic $\ad S$-modules. For any ${\pi_j}\neq t$, the subspace $(A_1)_{\pi_j}$ decomposes into a sum of irreducible and non isomorphic $\ad (S \oplus \mathbb{F}a)$-subspaces. Hence $L$ is of type \textbf{B-III}. According to Lemma~\ref{lema-reduction-1}, any ideal of $L$ decomposes as  $S'\oplus I_R$, where $S'$ is any subalgebra of $S$ such that $[S',A]=[S',A_1]\subseteq I_R$ and $I_R$ is any ideal of $L$ inside $A\oplus \mathbb{F}a$. Similar arguments to those given in the proof of Lemma~\ref{reduction-2} yield the following possibilities for the ideals of $L$:

          \begin{itemize}
            \item[(a)] $S'\oplus I_{\mathcal{J}}$ where $I_R=I_{\mathcal{J}}$ is an ideal of $L$ inside $\mathcal{J}(L)=A_0\oplus A_1$. Since $\mathcal{J}(L)^2=0$, $I_{\mathcal{J}}$ is just any $\ad (S\oplus \mathbb{F}a)$-invariant subspace of $\mathcal{J}(L)$. Note that in the case $S'\neq 0$, we have $[S', A_1]\neq 0$ because of $\Ssoc(L)=0$. In fact, $S=S'$ give us $A_1\subseteq I_{\mathcal{J}}$.
            \item[(b)] $S'\oplus (I_{\mathcal{J}} \oplus \mathbb{F}(x+a))$, with $I_{\mathcal{J}}$ as in item~(a) and $x\in A_0$ and $[a, \mathcal{J}(L)]\subseteq I_{\mathcal{J}}$. In this case, $A_0\oplus (\oplus_{\pi_j\neq t} (A_1)_{\pi_j})\subseteq I_{\mathcal{J}}$ and therefore $x=0$. Hence, $I_R=I_{\mathcal{J}} \oplus \mathbb{F}a$.
          \end{itemize}

          From (a) and (b), it is clear that Lie algebras of type \textbf{B-III} have a finite number of ideals. As upper bound we can provide $2^l\cdot(2^{s+k}+2^{k_0})$. Here $k_0$ is the number of irreducible $\ad S$-modules in the decomposition of $(A_1)_t$. The lower bound is clear.
    \item[$\bullet$] $L$ is as in (IV). Then $\Asoc(L)=\mathcal{J}(L)$ yields to the algebras of type \textbf{C}.
  \end{itemize}
  Finally, $\varphi(L)=0$ and $\Ssoc(L)\neq 0$ provides the Lie algebras of type \textbf{D}.
\end{proof}

Our final example displays the algebraic structure and the Hasse diagram of the lattice of ideals of Lie algebras with trivial Frattini subalgebra and number of ideals $\leq 10$.
\begin{example}The following list exhausts all the Lie algebras up to $10$ ideals and trivial Fattini subalgebra. The decomposition $V\oplus W$ will represent a direct sum as vector spaces and $V\boxplus W$ a direct sum as ideals. So, $V$ and $W$ will be Lie algebras in the latter case.
  \begin{figure}
    \centering
    \includegraphics[width=0.8\textwidth]{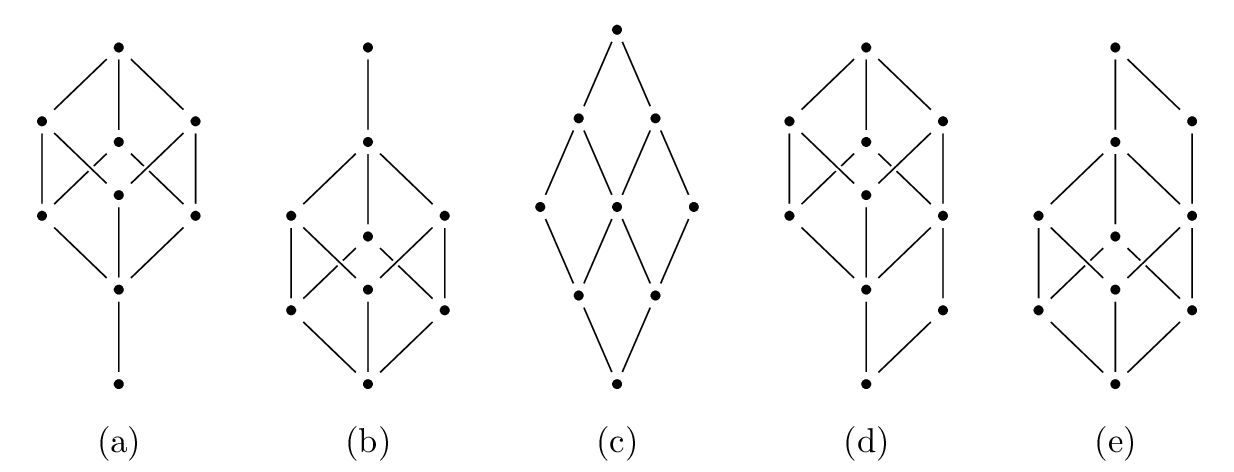}
    \caption{Some Hasse diagrams with more than 8 nodes.}
    \label{fig:HasseMore}
  \end{figure}
  \begin{itemize}
    \item[\emph{1.1}] \textbf{A}-type: \emph{$0$ and $\mathbb{F}z$. Hasse diagrams h1.1 and h2.1 in Figure~\ref{fig:allHasse}. }

    \item[\emph{1.2}] \textbf{A$\,\oplus\,$D}-type: $S_i$  simple Lie algebras.

          \begin{itemize}
            \item[\emph{1.2.1}] \emph{Any simple Lie algebra. Hasse diagram h2.1 in Figure~\ref{fig:allHasse}.} 

            \item[\emph{1.2.2}] \emph{$S_1\boxplus S_2$ and $\mathbb{F}z\boxplus S_1$.  Hasse diagram h4.2 in Figure~\ref{fig:allHasse}.}

            \item[\emph{1.2.3}] \emph{$S_1\boxplus S_2\boxplus S_3$ and $\mathbb{F}z\boxplus S_1\boxplus S_2$. Hasse diagram h8.15 in Figure~\ref{fig:allHasse}.}

          \end{itemize}
    \item[\emph{2.1}] \textbf{B-I}-type: $S_i$ simple and $V_{i,j}$ irreducible, non trivial and non isomorphic $S_i$-modules and $\mathbb{F}$ denotes the trivial $S_i$-module.
          \begin{itemize}
            \item[\emph{2.1.1}] \emph{$S_1\oplus V_{1,1}$, $S_1\oplus (V_{1,1}\oplus V_{1,2})$ and $S_1\oplus (V_{1,1}\oplus V_{1,2} \oplus V_{1,3})$. Hasse diagrams h3.1 and h5.3 in Figure~\ref{fig:allHasse} and (b) in Figure~\ref{fig:HasseMore}.}

            \item[\emph{2.1.2}] \emph{$S_1\oplus S_2 \oplus (V_{1,1}\otimes V_{2,1})$, $S_1\oplus S_2 \oplus (V_{1,2}\otimes V_{2,2}\oplus V_{2,1}\otimes V_{2,2})$. Hasse diagrams h5.2 and h7.8 in Figure~\ref{fig:allHasse}.} 

            \item[\emph{2.1.3}] \emph{$S_1\oplus S_2 \oplus (V_{1,1}\otimes \mathbb{F} \oplus V_{1,2}\otimes V_{2,2})$ and $S_1\oplus S_2 \oplus (V_{1,1}\otimes \mathbb{F} \oplus \mathbb{F}\otimes V_{2,2})$. Hasse diagrams h8.13 in Figure~\ref{fig:allHasse} and (c) in Figure~\ref{fig:HasseMore}.}
            \item[\emph{2.1.4}] \emph{$S_1\oplus S_2 \oplus S_3 \oplus V_{1,1}\otimes V_{2,1} \otimes V_{3,1}$. Hasse diagram (a) in Figure~\ref{fig:HasseMore}.}

          \end{itemize}

    \item[\emph{2.2}] \textbf{B-I$\,\boxplus\,$D}-type: \emph{$S\boxplus (S_1\oplus V_{1,1})$, $S\boxplus (S_1\oplus (V_{1,1}\oplus V_{1,2}))$ and $S\boxplus (S_1\oplus S_2 \oplus V_{1,1}\otimes V_{2,1})$, where $S$ is a simple ideal. Hasse diagrams h6.5 in Figure~\ref{fig:allHasse} and (e) and (d) in Figure~\ref{fig:HasseMore}.}

    \item[\emph{3.1}] \textbf{B-II}-type: \emph{$\mathbb{F}z\boxplus (S_1\oplus V_{1,1})$, $\mathbb{F}z\boxplus (S_1\oplus (V_{1,1}\oplus V_{1,2}))$ and $\mathbb{F}z\boxplus (S_1\oplus S_2 \oplus V_{1,1}\otimes V_{2,1})$. Hasse diagrams h6.5 in Figure~\ref{fig:allHasse} and (e) and (d) in Figure~\ref{fig:HasseMore}.}

    \item[\emph{3.2}] \textbf{B-II$\,\boxplus\,$D}-type: \emph{More than $10$ ideals.}

    \item[\emph{4.1}] \textbf{B-III}-type: Here $n$-$\ad(S\oplus \mathbb{F}a)$-irreducible means that the number of irreducible and non isomorphic $\ad(S\oplus \mathbb{F}a)$-subspaces is $n$.
          \begin{itemize}
            \item[\emph{4.1.1}] \emph{$S_1\oplus S_2 \oplus A_1\oplus \mathbb{F}a$, where $A_1=(A_1)_{\pi_1}$ is $1$-$\ad(S\oplus \mathbb{F}a)$-irreducible and $\pi_1(t)\neq t$. Hasse diagram (a) in Figure~\ref{fig:HasseMore}.}

            \item[\emph{4.1.2}] \emph{$S_1\oplus A_1\oplus \mathbb{F}a$, where $A_1=(A_1)_{\pi_1}$ is $n$-$\ad(S\oplus \mathbb{F}a)$-irreducible, $n=1,2$, and $\pi_1(t)\neq t$. Hasse diagrams h5.2 and h7.8 in Figure~\ref{fig:allHasse}.}

            \item[\emph{4.1.3}] \emph{$S_1\oplus A_1\oplus \mathbb{F}a$, where $A_1=(A_1)_{\pi_1}\oplus (A_1)_{\pi_2}$ and each $(A_1)_{\pi_j}$ is $1$-$\ad(S\oplus \mathbb{F}a)$-irreducible and $\pi_j(t)\neq t$. Hasse diagram h7.8 in Figure~\ref{fig:allHasse}.}

            \item[\emph{4.1.4}] \emph{$S_1\oplus A_1\oplus \mathbb{F}a$, where $A_1=(A_1)_{t}\oplus (A_1)_{\pi_1}$ and $\pi_1(t)\neq t$, $(A_1)_{t}$ is $\ad S$-irreducible and $(A_1)_{\pi_1}$ $1$-$\ad(S\oplus \mathbb{F}a)$-irreducible. Hasse diagram h8.13 in Figure~\ref{fig:allHasse}. }

            \item[\emph{4.1.5}] \emph{$S_1\oplus A_0\oplus A_1\oplus \mathbb{F}a$, where $A_0=(A_0)_{\pi_1}$ is $\ad a$-irreducible and $(A_1)_{\pi_2}$ is $1$-$\ad(S\oplus \mathbb{F}a)$-irreducible and $\pi_j(t )\neq t$. Hasse diagram h8.13 in Figure~\ref{fig:allHasse}.}
            \item[\emph{4.1.6}] \emph{$S_1\oplus A_0\oplus A_1\oplus \mathbb{F}a$, where $A_0=(A_0)_{\pi_1}$ is $\ad a$-irreducible and $(A_1)_t$ is $\ad S$-irreducible and $\pi_1(t)\neq t$. Hasse diagram (c) in Figure~\ref{fig:HasseMore}.}
          \end{itemize}
    \item[\emph{4.2}] \textbf{B-III$\,\boxplus\,$D}-type: \emph{$S\boxplus (S_1\oplus A_1\oplus \mathbb{F}a)$, where $A_1=(A_1)_{\pi_1}$ is $1$-$\ad(S\oplus \mathbb{F}a)$-irreducible, $\pi_1(t)\neq t$ and $S$ is a simple ideal. Hasse diagram (d) in Figure~\ref{fig:HasseMore}.}

    \item[\emph{5.1}] \textbf{C}-type: \emph{Split extensions $L(V,f,\pi,n)=V\oplus \mathbb{F} f$, where $V$ is a vector space and $f\colon V\to V$ a semisimple cyclic linear map with minimum polynomial either irreducible $\pi=\pi_1(t)$ ($n=1$) or product of two distinct irreducible, $\pi=\pi_1(t)\pi_2(t)$ ($n=2$), or three distinct irreducible $\pi=\pi_1(t)\pi_2(t)\pi_3(t)$ ($n=3$) and $\pi_j(t)\neq t$. Hasse diagrams h3.1 and h5.3 in Figure~\ref{fig:allHasse} and (b) in Figure~\ref{fig:HasseMore}.}

    \item[\emph{5.2}] \textbf{C$\,\boxplus\,$D}-type: \emph{$S\boxplus L(V,f,\pi,1)$, $S\boxplus L(V,f,\pi,2)$. Hasse diagrams h6.5 in Figure~\ref{fig:allHasse} and (e) in Figure~\ref{fig:HasseMore}.}\end{itemize}
\end{example}

\bibliography{main}

\end{document}